\documentclass[10pt]{amsart}
\usepackage[latin1]{inputenc}
\usepackage{mathptmx}
\newtheorem{theorem}{Theorem}[section]
\newtheorem{lemma}[theorem]{Lemma}
\newtheorem{corollary}[theorem]{Corollary}
\newtheorem{proposition}[theorem]{Proposition}

\theoremstyle{definition}
\newtheorem{remark}[theorem]{Remark}

\newtheorem{definition}[theorem]{Definition}
\newtheorem{example}[theorem]{Example}
\newtheorem{remark/example}[theorem]{Remark/Example}

\newtheorem{question}[theorem]{Question}

\let\oldlabel=\label
\def\prellabel{\marginparsep=1em\marginparwidth=44pt
 \def\label##1{\oldlabel{##1}\ifmmode\else\ifinner\else
 \marginpar{{\footnotesize\ \\ \tt
 ##1}}\fi\fi}}

\numberwithin{equation}{section}

\def\PP{ {\bf P} }
\def\NN{ {\bf N} }
\def\ZZ{ {\bf Z} }
\def\QQ{ {\bf Q} }

\def\CC{ {\bf C} }
\def\KK{ {\bf K} }

\def\LL{\mathcal L}

\def\F{\mathcal F}
\def\FF{\mathbf  F}
\def\MM{\mathbf  M}
\def\GG{\mathbf  G}

\newcommand{\ini}{\operatorname{in}}

\newcommand{\mm}{\operatorname{{\mathbf m}}}

\newcommand{\depth}{\operatorname{depth}}
\newcommand{\height}{\operatorname{height}}

\newcommand{\Tor}{\operatorname{Tor}}
\newcommand{\Ker}{\operatorname{Ker}}

\newcommand{\Image}{\operatorname{Image}}

\newcommand{\projdim}{\operatorname{pd}}
\newcommand{\reg}{\operatorname{reg}}
\newcommand{\gr}{\operatorname{gr}}

\newcommand{\Sym}{\operatorname{Sym}}

\newcommand{\Ann}{\operatorname{Ann}}
\newcommand{\HF}{\operatorname{HF}}

\newcommand{\Rate}{\operatorname{Rate}}
\newcommand{\slope}{\operatorname{slope}}
\newcommand{\lin}{\operatorname{lin}}
\newcommand{\ld}{\operatorname{ld}}

\numberwithin{equation}{section}

\begin{document}
\title{
Koszul algebras and regularity}
 \author{Aldo Conca, Emanuela De Negri, Maria Evelina Rossi}
\address{ Dipartimento di Matematica,
Universit\`a degli Studi di Genova, Italy} \email{conca@dima.unige.it, denegri@dima.unige.it, rossim@dima.unige.it}
 
\subjclass[2000]{}
\date{}
%\begin{abstract}
% \end{abstract}
\maketitle

\section{Introduction}
This is a survey paper on commutative Koszul algebras and Castelnuovo-Mumford regularities. Koszul algebras,  originally introduced  by Priddy \cite{P},  are graded $K$-algebras $R$ whose residue field $K$ has a linear free resolution as an $R$-module.  Here linear means that the non-zero entries of the matrices describing the maps in the resolution have degree $1$. For example, over the  symmetric algebra $S=\Sym_K(V)$ of a finite dimensional $K$-vector space $V$ the residue field $K$ is resolved by the Koszul complex which is linear. Similarly, for the exterior algebra   $\bigwedge_{K} V$ the residue field $K$ is resolved by the Cartan complex which is also linear.   In this paper we deal mainly with standard graded commutative $K$-algebras, that is, quotient rings of the polynomial ring $S$ by homogeneous ideals. 
The (absolute) Castelnuovo-Mumford regularity $\reg_S(M)$ is,  after Krull dimension and multiplicity, perhaps the most important invariant of  a finitely generated graded  $S$-module $M$, as it controls the vanishing of  both  syzygies and  the local cohomology modules of $M$.   By definition,  $\reg_S(M)$ is the least integer   $r$ such that the $i$-th syzygy module of $M$ is generated in degrees $\leq r+i$ for every $i$. By local duality, $\reg_S(M)$  can be characterized also as  the least number $r$ such that the local cohomology module $H^i_{\mm_S}(M)$  vanishes  in degrees $>r-i$ for every $i$. Analogously when  $R=S/I$  is a standard graded $K$-algebra and $M$ is a finitely generated graded  $R$-module one can define the relative Castelnuovo-Mumford regularity as the least integer $r$ such that the $i$-th syzygy module over $R$ of $M$ is generated in degrees $\leq r+i$ for every $i$. 
The main difference between the relative and the absolute regularity is that over $R$ most of the resolutions are infinite, i.e.~there are infinitely many syzygy modules,  and hence it is not at all clear whether $\reg_R(M)$ is finite. 
Avramov, Eisenbud  and Peeva gave in \cite{AP,AE} a beautiful characterization of the Koszul property in terms of the relative regularity: $R$ is Koszul iff  $\reg_R(M)$ is finite for every $M$ iff $\reg_R(K)$ is finite. 

From certain point of views,  Koszul algebras behave homologically as  polynomial rings. For instance $\reg_R(M)$ can be characterized in terms of regularity of truncated submodules (see \ref{pro28}).   On the other hand, ``bad" homological behaviors may occur over  Koszul algebras. For instance, modules might have  irrational Poinca\'re series over Koszul algebras. Furthermore, Koszul algebras appear quite frequently among the rings that are classically studied in commutative algebra, algebraic geometry and combinatorial commutative algebra. This mixture of similarities and differences with the polynomial ring and their  frequent appearance in classical constructions are some of the reasons that make Koszul algebras  fascinating, studied and beloved by commutative algebraists and algebraic geometers. In few words, a homological life is worth living in a Koszul algebra.  Of course there are other reasons for the popularity of Koszul algebras in the commutative and non-commutative setting, as, for instance,  Koszul duality, a  phenomenon that generalizes the duality between the symmetric and the exterior algebra, see \cite{BGS,BGSo,PP}. 

The structure of the paper is the following. Section \ref{s1} contains the characterization,  due to Avramov, Eisenbud and Peeva,   of Koszul algebras in terms of the finiteness of the regularity of modules, see \ref{AEP}. It contains also the definition of G-quadratic and LG-quadratic algebras and some fundamental questions concerning the relationships between these notions and the syzygies of Koszul algebras, see \ref{q4} and \ref{q5bis}. 

In Section \ref{howK} we present three elementary but powerful methods for proving that an algebra is Koszul: the existence of a Gr\"obner basis of quadrics, the transfer of Koszulness to quotient rings and Koszul filtrations. To illustrate these methods we apply them to Veronese algebras and Veronese modules. We prove that Veronese subalgebras of Koszul algebras are Koszul and that high enough Veronese subalgebras of any algebra are Koszul.  These and related results were proved originally in \cite{ ABH, BF, BM, CHTV, ERT}. 

Section \ref{stronKo} is devoted to two very strong versions of Koszulness: universally Koszul \cite{C1}  and absolutely Koszul \cite{IR}.  An algebra $R$ is universally Koszul if for every ideal $I\subset R$ generated by elements of degree $1$ one has $\reg_R(I)=1$. Given a  graded  $R$-module $M$ and $i\in \ZZ$ one defines $M_{\langle i\rangle}$ as the submodule of $M$ generated by the homogeneous component $M_i$  of degree $i$ of $M$. 
The $R$-module $M$ is componentwise linear if  $\reg_R(M_{\langle i\rangle})=i$ for every $i$ with $M_i\neq 0$.  The $K$-algebra $R$ is  absolutely Koszul if any finitely generated graded $R$-module $M$ has a componentwise linear $i$-th syzygy module for some $i\geq 0$.  Two major achievements are the complete characterization of the Cohen-Macaulay domains that are universally Koszul, see \cite{C1} or  \ref{uk2},  and the description of two classes of absolutely Koszul algebras, see \cite{IR} or  \ref{absKos}. We also present  some  questions related to these notions, in particular  \ref{q8} and \ref{q9}. 

In Section \ref{QuePro} we discuss some problems regarding the regularity of modules over Koszul algebras. Some are of computational nature, for instance \ref{q1}, and others are suggested by the analogy with the polynomial ring,  for example \ref{q12bis}. This section contains also some original results, in particular \ref{scarti1} and \ref{scarti2}, motivating the questions presented. 

Finally Section \ref{locca} contains a discussion on  local variants of the notion of Koszul algebra and the definition of Koszul modules. A local ring $(R, \mm, K) $ is called a Koszul ring if the associated graded ring  $\gr_{\mm}(R) $ is Koszul as a graded $K$-algebra.
The ring $R$ is called Fr\"oberg if its Poincar\'e series equals to $H_R(-z)^{-1}$, where   $H_R(z)$ denotes the Hilbert series of $R$. Any Koszul ring is Fr\"oberg. The converse holds in the graded setting and is unknown in the local case, see \ref{FK}. 
Large classes of local rings of almost minimal multiplicity are Koszul.  In \cite{HI} and \cite{IR}  a characterization of Koszulness of graded algebras is obtained in terms of the finiteness of the linear defect of the residue field, see \ref{infty}. It is an open problem whether the same characterization holds in the local case too, see \ref{loc1}.  

We thank Rasoul Ahangari, Lucho Avramov, Giulio Caviglia,  Ralf Fr\"oberg,  J\"uergen Herzog, Srikanth Iyengar, Liana \c Sega, Bart Snapp, Rekha Thomas and Matteo Varbaro  for useful discussion concerning the material presented in the paper.

 \section{Generalities}
\label{s1}
Let $K$ be a field and $R$ be a (commutative) standard graded $K$-algebra, that is a $K$-algebra with a decomposition  $R=\oplus_{i\in \NN} R_i$ (as an Abelian group) such that $R_0=K$,  $R_1$ is a finite dimensional $K$-vector space and $R_iR_j=R_{i+j}$ for every $i,j\in \NN$.  Let   $S$ be  the symmetric algebra over $K$ of $R_1$.  One has an induced surjection 
\begin{equation}
\label{canpre}
S=\Sym_K(R_1)\to R
\end{equation} 
of standard graded $K$-algebras. We call (\ref{canpre}) the canonical presentation of $R$. Hence $R$ is isomorphic (as a standard graded $K$-algebra) to $S/I$ where  $I$ is the kernel of (\ref{canpre}). In particular, $I$ is homogeneous and does not contain elements of degree $1$. We  say that $I$ defines $R$. Choosing a $K$-basis of $R_1$ the symmetric algebra  $S$ gets identified with  the polynomial ring  $K[x_1,\dots, x_n]$, with $n=\dim_K R_1$,  equipped with its standard graded structure (i.e. $\deg x_i=1$ for every $i$). Denote by $\mm_R$ the maximal homogeneous ideal of $R$. We may consider $K$ as a graded  $R$-module via the identification  $K=R/\mm_R$. \medskip 

\noindent {\bf Assumption}: With the exception of the last section, $K$-algebras  are always assumed to be standard graded,  modules and ideals are graded and finitely generated,  and  module homomorphisms have degree $0$. 
 \medskip 

For an $R$-module  $M=\oplus_{i\in \ZZ}  M_i$ we denote by $\HF(M,i)$ the Hilbert function of $M$ at $i$, that is,  $\HF(M,i)=\dim_K M_i$ and by $H_M(z)=\sum \dim_K M_i z^i\in \QQ[|z|][z^{-1}]$ the associated Hilbert series. 

Recall that a minimal graded free resolution of $M$ as an $R$-module is a complex of free $R$-modules 
$$\FF: \cdots \to  F_{i+1} \stackrel{\phi_{i+1}} \longrightarrow  F_{i}  \stackrel{\phi_{i}} \longrightarrow  F_{i-1} \to \cdots \to F_1\stackrel{\phi_{1}} \longrightarrow F_0\to 0$$
such that $H_i(\FF)=0$ for $i>0$ and $H_0(\FF)=M$, $\Image \phi_{i+1} \subseteq \mm_R F_i$ for every $i$. Such a resolution exists and it is unique up to an isomorphism of complexes, that is why we usually talk of ``the" minimal free (graded) resolution of $M$. 
 By definition, the $i$-th Betti number $\beta_i^R(M)$ of $M$ as an $R$-module is the rank of $F_i$. Each $F_i$ is a direct sum of shifted copies of $R$. The $(i,j)$-th graded Betti number $\beta_{ij}^R(M)$ of $M$ is the number of copies of $R(-j)$ that appear in $F_i$.  By construction one has $\beta_{i}^R(M)=\dim_K  \Tor^R_i(M,K)$  and $\beta_{ij}^R(M)=\dim_K  \Tor^R_i(M,K)_j$. The Poincar\'e series of $M$ is defined as: 
$$P_M^R(z)=\sum_i  \beta_i^R(M)z^i\in \QQ[|z|],$$ 
and its bigraded version is: 
$$P_M^R(s,z)=\sum_{i,j}  \beta_{i,j}^R(M)z^is^j\in \QQ[s][|z|].$$ 
We set 
$$t_i^R(M)=\sup\{ j : \beta_{ij}^R(M)\neq 0\}$$ 
where, by convention,  $t_i^R(M)=-\infty$ if $F_i=0$. By definition, $t_0^R(M)$ is the largest degree of a minimal generator of $M$. 
Two important invariants  that measure the ``growth" of the resolution of $M$ as an $R$-module are the projective dimension  
$$\projdim_R(M)=\sup\{ i : F_i\neq 0\}=\sup\{ i : \beta_{ij}^R(M)\neq 0 \mbox{ for some }j\}$$
and the Castelnuovo-Mumford regularity 
$$\reg_R(M)=\sup\{ j-i : \beta_{ij}^R(M)\neq 0\}=\sup\{ t_i^R(M)-i : i\in \NN\}.$$
 
We may as well consider $M$ as a module over the polynomial ring $S$ via (\ref{canpre}).  The regularity $\reg_S(M)$ of $M$ as an $S$-module has also a cohomological interpretation via local duality, (see  for example \cite{EG,BH}). Denoting by $H^i_{\mm_S}(M)$ the $i$-th local cohomology module with support on the maximal ideal of $S$ one has 
 $$\reg_S(M)=\max\{ j+i : H^i_{\mm_S}(M)_j\neq 0\}.$$
 Since $H^i_{\mm_R}(M)=H^i_{\mm_S}(M)$ for every $i$, nothing changes if on right hand side of the formula above we replace $S$ with $R$. 
 So $\reg_S(M)$ is in some sense the ``absolute" Castelnuovo-Mumford regularity. 
 Both $\projdim_R(M)$ and $\reg_R(M)$ can be infinite. 
 
 \begin{example}  Let  $R=K[x]/(x^3)$  and $M=K$.  Then $F_{2i}=R(-3i)$ and $F_{2i+1}=R(-3i-1)$ so that  $\projdim_R(M)=\infty$ and  $\reg_R(M)=\infty$. 
 \end{example} 
  
 Note that, in general, $\reg_R(M)$ is finite if $\projdim_R M$ is finite, but,  as we will see,  not the other way round. 

In the study of   minimal free resolutions over $R$,   the resolution $\KK_R$ of  the residue field $K$ as an $R$-module plays a prominent role. This is because   $\Tor_*^R(M,K)=H_*(M\otimes \KK_R)$ and hence $\beta_{ij}^R(M)=\dim_K H_i(M\otimes \KK_R)_j$.  A very important role is played also by the Koszul complex $K(\mm_R)$ on a minimal system of generators of the maximal ideal $\mm_R$ of $R$.

When is $\projdim_R(M)$ finite  for every $M$? The answer is given by one of the most classical results in commutative algebra:  the   Auslander-Buchsbaum-Serre Theorem.  We present here the graded variant of it  that can be seen as a strong version of the Hilbert syzygy theorem. 
\begin{theorem} 
\label{ABS}
The following conditions are equivalent:
\begin{itemize}
\item[(1)]  $\projdim_R M$ is finite for every $R$-module $M$, 
\item[(2)]  $\projdim_R K$ is finite, 
\item[(3)] $R$ is regular, that is, $R$ is a polynomial ring.  
\end{itemize}
When the conditions hold, then for every $M$ one has $\projdim_R M\leq \projdim_R K=\dim R$ and  the Koszul complex $K(\mm_R)$ resolves $K$ as an $R$-module, i.e. $\KK_R\cong K(\mm_R)$. 
\end{theorem} 
\begin{remark}
\label{koscomplex} 
The Koszul complex $K(\mm_R)$ has three important features:
  
(1)  it is finite, 

(2) it has an algebra structure. Indeed it  is a DG-algebra and this has important consequences such as the algebra structure on the Koszul cycles and Koszul homology. See \cite{A} for the definition (and much more) on  DG-algebras. 
 
 (3)  the matrices describing its differentials have non-zero entries only of degree $1$.   
 \end{remark}
 
When $R$ is not a polynomial ring  $\KK_R$   does not satisfies  condition (1) in \ref{koscomplex}. Can  $\KK_R$ nevertheless satisfy (2) or (3) of \ref{koscomplex}? 

For (2) the answer is yes: $\KK_R$ has always a DG-algebra structure. Indeed a theorem, proved independently by  Gulliksen and Schoeller (see \cite[6.3.5]{A}), 
asserts that $\KK_R$ is obtained by the so-called Tate construction. This procedure  starts from $K(\mm_R)$ and builds  $\KK_R$ by  ``adjoining variables to kill homology", while preserving the DG-algebra structure, see \cite[6.3.5]{A}. 

Algebras $R$ such that $\KK_R$ satisfies condition (3) in \ref{koscomplex} in above are called Koszul: 

\begin{definition} The $K$-algebra $R$ is Koszul if  the matrices describing  the differentials of $\KK_R$ have non-zero entries only of degree $1$, that is, $\reg_R(K)=0$ or, equivalently,  $\beta_{ij}^R(K)=0$ whenever $i\neq j$. 
\end{definition}

Koszul algebras were originally introduced  by Priddy \cite{P} in his study of  homological properties of graded (non-commutative) algebras arising from algebraic topology, leaving the commutative case  ``for the  interested reader".  In  the recent volume  \cite{PP}  Polishchuk and Positselski present various surprising aspects of Koszulness. 
We collect below  a list of important facts about Koszul commutative algebras. We always refer to the canonical presentation (\ref{canpre}) of $R$. 
First we introduce a definition.

\begin{definition}\label{Gquad} 
 We say that $R$ is G-quadratic if  its defining ideal $I$ has a Gr\"obner basis of quadrics with respect to some coordinate system of $S_1$ and some term order $\tau$ on $S$. 
 \end{definition} 

\begin{remark} 
\label{R16}
(1) If $R$ is Koszul, then $I$ is generated by quadrics (i.e. homogeneous polynomials of degree $2$). Indeed, the condition $\beta_{2j}^R(K)=0$ for every $j\neq 2$ is equivalent to the fact that $I$ is defined by quadrics. But there are algebras defined by quadrics that are not Koszul.  For example $R=K[x,y,z,t]/I$ with  $I=(x^2,y^2,z^2,t^2,xy+zt)$ has $\beta_{34}^R(K)=5$.  

(2) If $I$ is generated by monomials of degree $2$ with respect to some coordinate system of $S_1$, then a simple filtration argument that we reproduce in Section \ref{howK}, see \ref{monKos}, shows  that $R$ is Koszul in a very strong sense. 

(3) If $I$ is generated by a regular sequence of quadrics, then $R$ is Koszul. This follows from a result of Tate \cite{T} asserting that, if $R$ is a complete intersection, then $\KK_R$ is obtained by $K(\mm_R)$ by adding polynomial variables in homological degree $2$ to kill $H_1(K(\mm_R))$. 

(4) If $R$ is G-quadratic,  then $R$ is Koszul. This follows from  (2) and from  the standard deformation argument showing that $\beta_{ij}^R(K)\leq \beta_{ij}^{A}(K)$ with $A=S/\ini_\tau(I)$. 

(5) On the other hand there are Koszul algebras that are not G-quadratic. One notes that  an ideal defining a G-quadratic  algebra must contain quadrics of ``low" rank. For instance, if $R$ is Artinian and G-quadratic then its defining ideal must contain the square of a linear form. But most Artinian complete intersection of quadrics do not contain the square of a linear form. For example, $I=(x^2+yz, y^2+xz, z^2+xy)\subset \CC[x,y,z]$ is an Artinian complete intersection not containing the square of a linear form. Hence $I$ defines a  Koszul and not G-quadratic algebra. See \cite{ERT} for a general result in this direction. 

(6) The Poincar\'e series $P_K^R(z)$ of $K$ as an $R$-module can be irrational, see \cite{An}. However for a Koszul algebra $R$ one has 
\begin{equation} 
P_K^R(z)H_R(-z)=1 
\label{HilPoi} 
\end{equation}
and hence $P_K^R(z)$ is rational. Indeed the equality (\ref{HilPoi}) turns out to be equivalent to the Koszul property of $R$, \cite[1]{F}. A necessary (but not sufficient) numerical condition for $R$ to be Koszul is that the formal power series $1/H_R(-z)$ has non-negative coefficients (indeed positive unless $R$ is a polynomial ring). Another numerical condition is the following: expand $1/H_R(-z)$ as 
$$\frac{ \Pi_{h \in 2\NN+1}  (1+z^h)^{e_h}}{ \Pi_{h\in 2\NN+2} (1-z^h)^{e_h}}$$
with $e_h\in \ZZ$, see \cite[7.1.1]{A}. The numbers $e_h$ are the ``expected" deviations. If $R$ is Koszul then $e_h\geq 0$ for every $h$, (indeed $e_h>0$ for every $h$ unless $R$ is a complete intersection). For example, if $H(z)=1+4z+5z^2$  then the coefficient of $z^6$ in $1/H(-z)$  is negative and the third expected deviation is $0$. So for two reasons an algebra with Hilbert series $H(z)$, as the one in (1), cannot be Koszul.    \end{remark} 

The following  characterization of the Koszul property in terms of  regularity  is formally similar to the Auslander-Buchsbaum-Serre Theorem \ref{ABS}.

\begin{theorem}[Avramov-Eisenbud-Peeva]
\label{AEP}
The following conditions are equivalent:
\begin{itemize}
\item[(1)]  $\reg_R (M)$ is finite for every $R$-module $M$, 
\item[(2)]  $\reg_R (K)$ is finite, 
\item[(3)] $R$ is Koszul.  
\end{itemize} 
\end{theorem}  

 Avramov and Eisenbud proved in \cite{AE} that every module has finite regularity over a Koszul algebra. 
 Avramov and Peeva showed in \cite{AP} that if  $\reg_R (K)$ is finite then it must be $0$. Indeed they proved a more general result for graded algebras that are not necessarily standard. 

If $M$ is an $R$-module generated by elements of a given degree, say $d$, we say that it has a linear resolution over $R$  if $\reg_R(M)=d$.  For $q\in \ZZ$ we set $M_{\langle q\rangle}$ to be the submodule of $M$ generated by $M_q$ and set $M_{\geq q}=\oplus_{i\geq q} M_i$.  The module $M$ is said to be componentwise linear over $R$ if  $M_{\langle q\rangle}$ has a linear resolution for every $q$. 
The (absolute) regularity of a module can be characterized as follows: 
$$\begin{array}{rl}
\reg_S(M)=&\min\{ q\in \ZZ : M_{\geq q} \mbox{ has a linear resolution}\}\\
=&\min\{ q\geq t_0^S(M) :  M_{\langle q\rangle} \mbox{ has a linear resolution}\}
\end{array}$$

One of the motivations of Avramov and Eisenbud in \cite{AE} was to establish a similar characterization for the relative regularity over a Koszul algebra. They proved: 

\begin{proposition} 
\label{pro28}
Let $R$ be a Koszul algebra  and $M$ be an $R$-module.   Then: 
$$\reg_R(M)\leq  \reg_S(M)$$ and 
$$\begin{array}{rl}
\reg_R(M)=&\min\{ q\in \ZZ : M_{\geq q} \mbox{ has a linear $R$-resolution}\}\\
=&\min\{ q\geq t_0^R(M) :  M_{\langle q\rangle} \mbox{ has a linear $R$-resolution}\}
\end{array}$$
\end{proposition}

Another invariant that measures the growth of the degrees of the syzygies of a module  is the slope: 
$$\slope_R(M)=\sup\{\frac{ t^R_i(M)-t^R_0(M)}{i} : i>0\}.$$
A useful feature of the slope is that it is finite (no matter if $R$ is Koszul or not). Indeed with respect to the canonical presentation (\ref{canpre}) one has: 
$$\slope_R(M)\leq \max\{ \slope_S(R), \slope_S(M)\}$$
see \cite[1.2]{ACI},  and the right hand side is finite since $S$ is a polynomial ring. 
Backelin defined in \cite{B} the (Backelin) rate of $R$ to be
$$\Rate(R)=\slope_R(\mm_R)$$
as a measure of the failure of the Koszul property.  By the very definition, one has $\Rate(R)\geq 1$ and  $R$ is Koszul if and only if $\Rate(R)=1$. 

We close the section with a technical lemma:   

\begin{lemma}
\label{blabla}
(1)   Let $0\to M_1\to M_2\to M_3\to 0$
be a short exact sequence of $R$-modules. Then one has
 $$\begin{array}{rl}
 \reg_R(M_1)\leq &\!\!\!\!\!\!  \max\{ \reg_R(M_2),  \reg_R(M_3)+1\},\\
 \reg_R(M_2)\leq & \!\!\!\!\!\! \max\{ \reg_R(M_1),  \reg_R(M_3)\},\\
 \reg_R(M_3)\leq &\!\!\!\!\!\!  \max\{ \reg_R(M_1)-1,  \reg_R(M_2)\}.
 \end{array}
 $$
(2) Let $$\MM: \cdots \to M_i\to\cdots\to M_2\to M_1\to M_0\to 0$$ be a complex  of $R$-modules. Set $H_i=H_i(\MM)$. Then for every $i\geq 0$ one has
 $$t_i^R(H_0)\leq \max\{a_i,b_i\}$$ where 
 $a_i=\max\{ t_j^R(M_{i-j}) : j=0,\dots, i\}$
 and  $b_i=\max\{ t_j^R(H_{i-j-1}) : j=0,\dots,i-2\}$. 
 
\noindent Moreover one has  
$$\reg_R(H_0)\leq \max\{a,b\}$$
 where $a=\sup\{ \reg_R (M_j)-j : j\geq 0\}$ and $b=\sup\{ \reg_R (H_j)-(j+1) : j\geq 1\}$.
 \end{lemma} 
 \begin{proof} (1) follows immediately by considering the long exact sequence obtained by applying $\Tor(K,-)$. For (2) one breaks the complex into short exact sequences and proves by induction on $i$  the inequality for $t_i^R(H_0)$.   Then one deduces the second inequality by translating the first  into a statement about regularities. 
  \end{proof} 

 \medskip
 
We collect below some problems  about the Koszul property and the existence of Gr\"obner bases of quadrics.
Let us recall the following
 
 \begin{definition}
 \label{LG-quad}
 A $K$-algebra $R$ is LG-quadratic if there exists a G-quadratic algebra $A$ and a regular sequence of linear forms $y_1,\dots,y_c$ such that $R\simeq A/(y_1,\dots,y_c)$. 
 \end{definition} 
 
 We have the following implications: 
 \begin{equation}
 \label{GLG}
   \mbox{G-quadratic } \Rightarrow \mbox{LG-quadratic } \Rightarrow \mbox{Koszul}  \Rightarrow \mbox{quadratic }
   \end{equation}
 
As discussed in \ref{R16} the third  implication in  (\ref{GLG}) is  strict.  The following remark, due to Caviglia, in connection with \ref{R16}(5) shows that also the first implication in  (\ref{GLG}) is strict. 
 \begin{remark} 
 Any complete intersection $R$ of quadrics is LG-quadratic. 
 
 \noindent Say $R=K[x_1,\dots,x_n]/(q_1,\dots,q_m)$ then set 
 $$A=R[y_1,\dots,y_m]/(y_1^2+q_1,\dots,y_m^2+q_m)$$ and note that $A$ is G-quadratic for obvious reasons and $y_1,\dots,y_m$ is a regular sequence in $A$ by codimension considerations. 
 \end{remark} 
 
 But we do not know an example of a Koszul algebra that is not LG-quadratic. So we ask: 
 
 \begin{question}
\label{q4} 
Is any Koszul algebra LG-quadratic?  
\end{question}

Our feeling is that the answer should be negative. But how can we exclude that a Koszul algebra is LG-quadratic? One can look at the $h$-vector (i.e. the numerator of the Hilbert series)  since it is invariant under Gr\"obner deformation and modifications as the one involved in the definition  of LG-quadratic. Alternatively one can look at syzygies over the polynomial ring because they can only grow under such operations. These observations lead  to a new question: 

 \begin{question}
\label{q5} 
Is the $h$-vector of any  Koszul algebra $R$ the $h$-vector of an algebra defined by quadratic monomials? And, if yes, does there exist an algebra $A$ with  quadratic monomial relations, $h$-vector equal to that of $R$ and satisfying $\beta^S_{ij}(R)\leq \beta^{T}_{ij}(A)$ for every $i$ and $j$? Here $S$ and $T$ denote the polynomial rings canonically projecting onto  $R$ and $A$.   \end{question}

 A negative answer to \ref{q5} would imply a negative answer to \ref{q4}. Note that any $h$-vector of an algebra defined by quadratic monomials is also the $h$-vector of an algebra defined by square-free quadratic monomials (by using the  polarization process). The simplicial complexes associated to square-free quadratic monomial ideals are  called flag. There has been a lot of activity concerning combinatorial properties and characterizations of $h$-vectors and $f$-vectors  of  flag simplicial complexes, see \cite{CV} for recent results and for a survey of what is known and conjectured. Here we just mention that  Frohmader has  proved in  \cite{Fm}  a  conjecture of Kalai asserting that the  $f$-vectors of flag simplicial complexes  are $f$-vectors of balanced simplicial complexes.  
 
 Regarding the inequality for  Betti numbers in \ref{q5},  LG-quadratic algebras $R$ satisfy the following restrictions  
 \begin{itemize} 
 \item[(1)] $t_i^S(R)\leq 2i$,
 \item[(2)] $t_i^S(R)<2i$ if $t_{i-1}^S(R)<2(i-1)$, 
 \item[(3)] $t_i^S(R)<2i$ if $i>\dim S-\dim R$, 
 \item[(4)] $\beta^S_i(R)\leq \binom{\beta_1^S(R)}{i}$,
 \end{itemize} 
deduced from the deformation to the  (non-minimal) Taylor resolution of quadratic monomial ideals, see for instance \cite[4.3.2]{MS}.    
As shown in   \cite{ACI} the same restrictions are satisfied by any Koszul algebra,  with the exception of possibly  (4).  So we ask: 

 \begin{question}
\label{q5bis} 
Let $R$ be a Koszul algebra quotient of the polynomial ring $S$. Is it true that   $\beta^S_i(R)\leq \binom{\beta_1^S(R)}{i}$? 
\end{question}

 It can be very difficult to decide whether a given Koszul algebra is G-quadratic. In the nineties Peeva and Sturmfels asked whether the coordinate ring 
 $$PV=K[x^3,x^2y,x^2z,xy^2,xz^2,y^3,y^2z,yz^2,z^3]$$ 
 of  the pinched Veronese  is Koszul.  For about a decade this  was a benchmark example for  testing new techniques for proving Koszulness. In 2009 Caviglia \cite{Ca2}  gave the first proof of the Koszulness of $PV$. Recently a new one has been presented in \cite{CC} that applies also to a larger family of rings including all the general projections to $\PP^8$ of the Veronese surface in $\PP^9$. 
 The problem remains to decide whether: 
 \begin{question}
\label{q6} 
Is $PV$ G-quadratic?
  \end{question}
The answer is negative if one considers the toric coordinates only  (as it can be checked by computing the associated  
Gr\"obner fan using  CaTS \cite{Cats}),  but unknown in general. There are plenty of quadratic monomial ideals defining algebras with the Hilbert function of $PV$ and larger Betti numbers.  
  
The algebra $PV$ is generated by all monomials in $n=3$ variables of degree $d=3$ that are supported on at most $s=2$ variables. By varying  the  indices $n,d,s$ one gets a  family of pinched Veronese algebras $PV(n,d,s)$ and it is natural to ask: 

\begin{question}
\label{q7} 
For which values of $n,d,s$ is $PV(n,d,s)$ quadratic or Koszul?
\end{question}
Not all of them are quadratic, for instance $PV(4,5,2)$ is not. Questions as \ref{q7} are very common in the literature: in a family of algebras one asks which ones are quadratic or Koszul or if quadratic and Koszul are equivalent properties for the algebras in the family.  For example in \cite[6.10]{CRV} the authors ask: 

\begin{question}
\label{q7bis} 
Let $R$ be a  quadratic Gorenstein algebra with Hilbert series $1+nz+nz^2+z^3$. Is $R$ Koszul?  
\end{question}
For $n=3$ the answer is obvious as $R$ must be a complete intersection of quadrics and for $n=4$ the answer is positive by \cite[6.15]{CRV}. See \ref{1nn1} for results concerning this family of  algebras.

\section{How to prove that an algebra is Koszul?}\label{howK}
 To prove that an algebra  is Koszul is usually a difficult task.  There are examples, due to Roos,  showing that a sort of Murphy's  law (anything that can possibly go wrong, does) holds in this context. Indeed there  exists a family of quadratic algebras  $R(a)$ depending on an integer $a>1$ such that  the Hilbert series  of $R(a)$ is $1+6z+8z^2$ for every $a$. Moreover $K$ has a linear resolution for $a$ steps and a non-linear syzygy in homological position $a+1$, see \cite{R1}. So there is no statement of the kind: if $R$ is an algebra with Hilbert series  $H$ then there is a number $N$ depending on $H$ such that if the resolution of $K$ over $R$ is linear for $N$ steps it will be linear forever. 
 
 The goal of this section is to present some techniques to prove that an algebra is Koszul (without pretending they  are the most powerful or interesting).  For the sake of illustration we will apply these techniques to discuss the Koszul properties of Veronese algebras and modules.  The material we present is taken from various sources, see \cite{ABH,ACI, B, BF, BM, BCR, CDR,CHTV, CRV, CTV, HHR, ERT,S}. 
 
 \subsection{Gr\"obner basis of quadrics}
 \ 
 
The simplest way to prove that an algebra is Koszul is to show that it is G-quadratic. A weak point of this prospective is that Gr\"obner bases refer to a system of coordinates and a term order.  As said earlier, not all the  Koszul algebras are G-quadratic. On the other hand many of the classical constructions in  commutative algebra and algebraic geometry  lead to algebras that have a privileged, say natural,  system of coordinates. For instance, the coordinate ring of the Grassmannian comes equipped with the Pl\"ucker coordinates.  Toric varieties come with their toric coordinates. So one looks for a Gr\"obner basis of quadrics with respect to the natural system of coordinates. It turns out that many of the classical algebras (Grassmannian, Veronese, Segre, etc..) do have Gr\"obner bases of quadrics in the natural system of coordinates. Here we treat in details the Veronese case: 

\begin{theorem}
\label{GBqVero}
Let $S=K[x_1,\dots,x_n]$ and $c\in \NN$. Then the Veronese subring $S^{(c)}=\oplus_{j\in \NN}  S_{jc}$ is defined by a Gr\"obner basis of quadrics.  
\end{theorem} 
\begin{proof} 
For $j\in \NN$ denote by $M_j$ the set of monomials of degree $j$ of $S$.  Consider  $T_c=\Sym_K(S_c)=K[t_m : m\in M_c]$  and the surjective map $\Phi: T_c\to S^{(c)}$ of $K$-algebras with $\Phi(t_m)=m$ for every $m\in M_c$. For every monomial $m$ we set 
$\max(m)=\max\{ i : x_i|m\}$ and $\min(m)=\min\{ i : x_i|m\}$. Furthermore for monomials $m_1,m_2\in M_c$ we set $m_1\prec m_2$ if $\max(m_1)\leq \min(m_2)$. Clearly $\prec$ is a transitive (but not reflexive) relation. We say that $m_1,m_2\in M_c$ are incomparable if $m_1\not\prec m_2$ and  $m_2\not\prec m_1$ and that are comparable otherwise. For a pair of incomparable elements $m_1,m_2\in M_c$,
let $m_3,m_4\in M_c$ be the uniquely determined elements in $M_c$ such that $m_1m_2=m_3m_4$ and $m_3\prec m_4$. Set 
$$F(m_1,m_2)=t_{m_1}t_{m_2}-t_{m_3}t_{m_4}. $$
By construction $F(m_1,m_2)\in \Ker \Phi$ and we claim that the set of the $F(m_1,m_2)$'s  is a Gr\"obner basis of $\Ker \Phi$ with respect to any term order $\tau$ of $T_c$ such that $\ini_\tau(F(m_1,m_2))=t_{m_1}t_{m_2}$. Such a term order exists: order the  $t_m's$  totally as follows
$$t_{u}\geq t_{v}  \mbox{ iff } u\geq v \mbox{ lexicographically }$$  
and then consider the degree reverse lexicographic term order associated to that total order. 
Such a term order has the required property as it is easy to see. It remains to prove that the $F(m_1,m_2)$'s  form a Gr\"obner basis of $\Ker \Phi$. 
Set 
$$U=( t_{m_1}t_{m_2} : m_1,m_2\in M_c \mbox{  are incomparable })$$
By construction we have $U\subset \ini_\tau(\Ker \Phi)$ and we have to prove equality. We do it by checking that the two associated quotients  have the same Hilbert function. The inequality  $\HF(T_c/\Ker \Phi,i)\leq  \HF(T_c/U,i)$ follows from the inclusion of the ideals. For the other note that 
$$\HF(T_c/\ini_\tau(\Ker \Phi),i)=\HF(T_c/\Ker \Phi,i)=\HF(S^{(c)},i)=\# M_{ic}$$
The key observations are: \medskip

(1)  a monomial in the $t$'s, say $t_{m_1}\cdots t_{m_i}$, is not in $U$  if (after a permutation) $m_1\prec m_2\prec \cdots\prec m_i$,\medskip 

(2)  every monomial $m\in M_{ic}$ has a uniquely determined decomposition $m=m_1\cdots m_i$ with $m_1\prec m_2\prec \cdots\prec m_i$. \medskip 

This implies that 
$$\HF(T_c/U,i)\leq \# M_{ic}$$ 
proving the desired assertion. 
\end{proof} 

\subsection{Transfer of Koszulness} 
\

Let $A$ be a $K$-algebra $A$ and   $B=A/I$ a quotient of it.  Assume one of the two algebras is  Koszul. What do we need to  know about the relationship  between $A$ and $B$ to conclude that the other algebra  is  Koszul too? Here is an answer: 

\begin{theorem} 
\label{Giux1}
Let $A$ be a $K$-algebra and $B$ be a quotient of $A$. 
\begin{itemize} 
\item[(1)] If  $\reg_A(B)\leq 1$ and $A$ is Koszul, then $B$ is Koszul. 
\item[(2)] If  $\reg_A(B)$ is finite  and $B$ is Koszul, then $A$ is Koszul. 
\end{itemize} 
\end{theorem} 

The theorem is a corollary of the following: 

\begin{proposition} 
\label{Giux2}
Let $A$ be a $K$-algebra and $B$ a quotient algebra of $A$.
Let $M$ be a $B$-module. Then: 
\begin{itemize}
\item[(1)]  $\reg_A(M)\leq \reg_B(M)+\reg_A(B)$. 
\item[(2)] if $\reg_A(B)\leq 1$ then   $\reg_B(M)\leq \reg_A(M)$. 
\end{itemize} 
\end{proposition} 
\begin{proof} One applies \ref{blabla}(2)  to the minimal free resolution $\FF$ of $M$ as a $B$-module and one has: 
$$\reg_A(M)\leq \sup\{ \reg_A (F_j) -j : j\geq 0\}.$$
Since $\reg_A (F_j)=\reg_A(B)+t_j^B(M)$, we can conclude that (1) holds. 

For (2) it is enough to prove that the inequality 
$$t_i^B(M)-i\leq \max\{ t_j^A(M)-j : j=0,\dots, i\}$$ holds for every $i$. We argue by induction on $i$; the case $i=0$ is obvious because $t_0^A(M)=t_0^B(M)$. Assume $i>0$ and take a minimal presentation of $M$ as a $B$-module
$$0\to N\to F\to M\to 0$$
where $F$ is $B$-free. Since $t^B_i(M)=t_{i-1}^B(N)$, by induction we have: 
$$t_i^B(M)-i=t_{i-1}^B(N)-i\leq \max\{ t_j^A(N)-j-1 : j=0,\dots,i-1\}$$
Since $t_j^A(N)\leq \max\{ t_j^A(F), t_{j+1}^A(M)\}$ and $t_j^A(F)=t_j^A(B)+t_0^A(M)\leq j+1+t_0^A(M)$ we may conclude that the desired inequality holds. 
 \end{proof} 

\begin{proof}[Proof of \ref{Giux1}] (1) Applying \ref{Giux2}(1) with $M$ equal to $K$ one has that $\reg_A(K)\leq \reg_A(B)$ which is finite by assumption. It follows then from \ref{AEP} that $A$ is Koszul. 
For (2) one applies  \ref{Giux2}(2) with $M=K$ and one gets $\reg_B(K)\leq \reg_A(K)$ which is $0$ by assumption, hence $\reg_B(K)=0$ as required. 
\end{proof} 

\begin{lemma}
\label{mM-reg}
Let $R$ be Koszul algebra and $M$ be an $R$-module. Then $$\reg_R(\mm_RM)\leq \reg_R(M)+1.$$
 In particular, $\reg_R(\mm_R^u)=u$, (unless $\mm_R^u=0$) that is, $\mm_R^u$ has a linear resolution for every $u\in \NN$. 
\end{lemma} 
\begin{proof} Apply \ref{blabla} to the short exact sequence 
$$0\to \mm_RM \to M\to M/ \mm_RM\to 0$$
 and use  the fact that $M/ \mm_RM$ is a direct sum of copies of $K$ shifted  at most by $-t_0^R(M)$.  
\end{proof}
 
 We apply now \ref{Giux1} to prove that the Veronese subrings of a Koszul algebra are Koszul.  
 
 Let  $c\in \NN$ and $R^{(c)}=\oplus_{j\in \ZZ}  R_{jc}$ be  the $c$-th Veronese subalgebra of $R$.  Similarly one defines $M^{(c)}$ for every  $R$-module $M$. The formation of the $c$-th Veronese submodule is an exact functor from the category of  $R$-modules  to  the category of graded $R^{(c)}$-modules (recall that, by convention,  modules are graded and maps are homogeneous of degree $0$). 
 For $u=0,\dots, c-1$ consider the Veronese submodules $V_u=\oplus_{j\in \ZZ}  R_{jc+u}$. Note that $V_u$ is an $R^{(c)}$-module generated in degree $0$ and that for $a\in \ZZ$ one has 
 $$R(-a)^{(c)}=V_u(-\lceil a/c \rceil )$$
  where $u=0$ if $a\equiv 0$  mod$(c)$ and $u=c-r$ if  $a\equiv r$  mod$(c)$ and $0<r<c$.

 \begin{theorem}  
 \label{cc1}
 Let $R$ be Koszul.  Then $R^{(c)}$ is Koszul and $\reg_{R^{(c)}}(V_u)=0$ for every $u=0,\dots, c-1$.    
 \end{theorem} 
 \begin{proof} Set $A=R^{(c)}$.  First we prove first that $\reg_{A}(V_u)=0$ for every $u=0,\dots, c-1$. To this end, we prove by induction on $i$ that $t_i^{A}(V_u)\leq i$ for every $i$. The case $i=0$ is obvious. So assume $i>0$. Let $M=\mm_R^u(u)$. By \ref{mM-reg} and by construction we have $\reg_R(M)=0$ and $M^{(c)}=V_u$. Consider the minimal free resolution $\FF$ of $M$ over $R$ and apply the functor $-^{(c)}$. We get a complex $\GG=\FF^{(c)}$ of $A$-modules such that $H_0(\GG)=V_u$,  $H_j(\GG)=0$ for $j>0$ and $G_j=F_j^{(c)}$ is a direct sum of copies of $R(-j)^{(c)}$.  Applying \ref{blabla} 
we get $t^A_i(V_u)\leq \max\{ t^A_{i-j}(G_j) : j=0,\dots, i\}$. Since $G_0$ is $A$-free we have $t_i^A(G_0)=-\infty$. For  $j>0$ we have  $R(-j)^{(c)}=V_w(-\lceil j/c \rceil)$ for some number $w$ with $0\leq w<c$. Hence, by induction, $t_{i-j}^A(G_j)\leq i-j+\lceil j/c \rceil\leq i$. Summing up, 
$$t^A_i(V_u)\leq  \max\{  i-j+\lceil j/c\rceil  :  j=1,\dots, i\}=i.$$
 In order to prove that $A$ is Koszul we consider the minimal free resolution $\FF$ of $K$ over $R$  and apply $-^{(c)}$. We get a complex $\GG=\FF^{(c)}$ of $A$-modules such that $H_0(\GG)=K$, $H_j(\GG)=0$ for $j>0$ and $G_j=F_j^{(c)}$ is a direct sum of copies of  $V_u(-\lceil j/c \rceil )$. Hence $\reg_A ( \GG_j)=\lceil j/c \rceil$ and applying \ref{blabla} we obtain 
 $$\reg_A(K)\leq \sup \{ \lceil j/c \rceil-j : j\geq 0\}=0.$$ 
  \end{proof}
 
 We also have: 
 
 \begin{theorem} 
 \label{cc2}
 Let $R$ be a $K$-algebra, then the Veronese subalgebra $R^{(c)}$ is Koszul for $c\gg 0$. More precisely, if $R=A/I$ with $A$ Koszul, then $R^{(c)}$ is Koszul for every $c\geq \sup\{ t_i^A(R)/(1+i)  : i\geq 0\}$.  
  \end{theorem} 
 
 \begin{proof} Let $\FF$ be the minimal free resolution of $R$ as an $A$-module.  Set $B=A^{(c)}$ and note that $B$ is Koszul because of \ref{cc1}. Then $\GG=\FF^{(c)}$ is a complex of $B$-modules such that $H_0(\GG)=R^{(c)}$,  $H_j(\GG)=0$ for $j>0$. Furthermore $G_i=F_i^{(c)}$ is a direct sum of shifted copies of the Veronese submodules $V_u$. Using \ref{cc1} we get  the bound $\reg_B(G_i)\leq \lceil t_i^A(R)/c \rceil$. Applying \ref{blabla} we get 
  $$\reg_B (R^{(c)})\leq \sup \{ \lceil t_i^A(R)/c \rceil -i : i\geq 0\}.$$
  Hence for $c\geq \sup\{ t_i^A(R)/(1+i)  : i\geq 0\}$ one has $\reg_B (R^{(c)})\leq 1$ and we conclude from \ref{Giux1} that $R^{(c)}$ is Koszul. 
  \end{proof} 
  
  \begin{remark} 
 (1) Note that the number $\sup\{ t_i^A(R)/(1+i)  : i\geq 0\}$  in \ref{cc2} is finite. For instance it is less than or equal to  $(\reg_A(R)+1)/2$ which is finite because $\reg_A(R)$ is finite. Note however that $\sup\{ t_i^A(R)/(1+i)  : i\geq 0\}$ can be much smaller than  $(\reg_A(R)+1)/2$; for instance if $R=A/I$ with $I$ generated by a regular sequence of $r$ elements of degree $d$, then $t_i^A(R)=id$ so that $\reg_A(R)=r(d-1)$ while  $\sup\{ t_i^A(R)/(1+i)  : i\geq 0\}=dr/(r+1)$. 
 
 (2) In particular, if we take the canonical presentation $R=S/I$ (\ref{canpre}), then we have that $R^{(c)}$ is Koszul if $c\geq \sup\{ t_i^S(R)/(1+i)  : i\geq 0\}$. In \cite[2]{ERT} it is proved that if $c\geq (\reg_S(R)+1)/2$, then $R^{(c)}$ is even G-quadratic. See \cite{Sh} for other interesting results in this direction. 

(3) Backelin proved in \cite{B} that $R^{(c)}$ is Koszul if $c\geq \Rate(R)$.  

(4) The proof of \ref{cc2} shows also that $\reg_{A^{(c)}} (R^{(c)})=0$ if $c\geq \slope_A(R)$.    
  \end{remark}

\subsection{Filtrations}
\

  Another tool for proving that an algebra is Koszul is a ``divide and conquer" strategy that can be formulated in various technical forms, depending on the  goal one has in mind. We choose the following: 

\begin{definition} A Koszul filtration of a $K$-algebra $R$ is a set $\F$ of ideals of $R$ such that: 
\begin{itemize} 
\item[(1)]  Every ideal $I\in \F$ is generated by elements of degree $1$.
\item[(2)] The zero ideal $0$ and the maximal ideal $\mm_R$ are in $\F$. 
\item[(3)] For every $I\in \F$, $I\neq 0$, there exists $J\in \F$ such that $J\subset I$,  $I/J$ is cyclic and $\Ann(I/J)=J:I\in \F$. 
\end{itemize}
\end{definition} 

By the very definition a Koszul filtration must contain a complete flag of $R_1$, that is, an increasing sequence $I_0=0\subset  I_1\subset \cdots \subset I_{n-1}\subset I_n=\mm_R$ such that $I_i$ is minimally generated by $i$ elements of degree $1$. The case where $\F$ consists of just a single flag deserves a name: 

\begin{definition} 
 A Gr\"obner flag for $R$ is a Koszul filtration that consists of a single complete flag of $R_1$. In other words, $\F=\{I_0=0\subset I_1\subset \cdots \subset I_{n-1}\subset I_n=\mm_R\}$ with  $I_{i-1}:I_i\in \F$ for every $i$.
\end{definition} 

One has: 

\begin{lemma} 
\label{filtflag} Let $\F$ be a Koszul filtration for $R$. Then one has: 
\begin{itemize} 
\item[(1)]  $\reg_R (R/I)=0$ and $R/I$ is Koszul for every $I\in \F$.   
\item[(2)]  $R$ is Koszul. 
\item[(3)] If $\F$ is a Gr\"obner flag, then $R$ is G-quadratic. 
\end{itemize}
\end{lemma} 
\begin{proof}  (1) and (2): One easily proves  by induction on $i$ and on the number of generators of $I$  that $t_i^R(R/I)\leq i$ for every $i$ and $I\in \F$. This implies that $R$ is Koszul (take $I=\mm_R$) and that $\reg_R (R/I)=0$, hence $R/I$ is Koszul by \ref{Giux1}. 

(3) We just sketch the argument: let $x_1,\dots,x_n$ be a basis for the flag, i.e.  $\F=\{I_0=0\subset I_1\subset \cdots \subset I_{n-1}\subset I_n=\mm_R\}$ and $I_i=(x_1,\dots,x_i)$ for every $i$. For every $i$ there exists $j_i\geq i$ such that 
$(x_1,\dots,x_{i}):x_{i+1}=(x_1,\dots, x_{j_i})$. For every $i<h\leq j_i$ the assertion $x_hx_{i+1}\in (x_1,\dots,x_{i})$ is turned into a quadratic equation in the defining ideal of $R$. The claim is that these quadratic equations form a Gr\"obner basis with respect to a term order that selects $x_hx_{i+1}$ as leading monomial. To prove the claim one shows that the  identified  monomials  define an algebra, call it $A$, whose  Hilbert function equals that of $R$. This is done by showing that the numbers $j_1,\dots, j_n$ associated to the flag of $R$ determine the Hilbert function of $R$ and then by showing that also $A$ has a Gr\"obner flag with associated numbers $j_1,\dots, j_n$. 
\end{proof}  

There are Koszul algebras without Koszul filtrations and G-quadratic algebras without Gr\"obner flags, see the examples given in \cite[pg.100 and 101]{CRV}. Families of algebras having Koszul filtrations or Gr\"obner flags are described in \cite{CRV}. For instance it is proved that  the coordinate ring of a set of  at most $2n$ points in $\PP^n$ in general linear position has a Gr\"obner flag,  and that the general Gorenstein Artinian algebra with socle in degree $3$ has a Koszul filtration. The results for points in  \cite{CRV} generalize results of \cite{CTV,Ke} and are generalized in \cite{Po}.   Filtrations of more general type are used in \cite{CDR} to control  the Backelin  rate of coordinate rings of sets of points in the projective space. 
 
 The following notion is  very natural for algebras with  privileged coordinate systems (e.g. in the the toric case). 

\begin{definition} An algebra $R$ is said to be strongly Koszul if there exists a basis $X$ of $R_1$ such that for every $Y\subset X$ and for every $x\in X\setminus Y$ there exists $Z\subseteq  X$ such $(Y):x=(Z)$. 
\end{definition} 

Our definition of strongly Koszul is slightly different than the one given in \cite{HHR}. In \cite{HHR} it is assumed that the  basis $X$ of $R_1$ is totally ordered  and in the definition one adds the requirement that $x$ is larger than every element in $Y$. 
  
\begin{remark} If $R$ is strongly Koszul with respect to a basis $X$ of $R_1$ then the set $\{ (Y) : Y\subseteq X\}$ is obviously a Koszul filtration. \end{remark} 

We have: 
\begin{theorem}
\label{sKos}
Let $R=S/I$ with $S=K[x_1,\dots,x_n]$ and $I\subset S$ an ideal generated by monomials of degrees $\leq d$. 
Then $R^{(c)}$ is strongly Koszul for every $c\geq d-1$.
\end{theorem}
\begin{proof} In the proof we use the following basic facts. \medskip 

\noindent Fact (1):  if $m_1,\dots,m_t, m$ are monomials of $S$,  then 
$(m_1,\dots,m_t):_Sm$ is generated by the monomials $m_i/\gcd(m_i,m)$ for  $i=1,\dots,t.$
\medskip 

\noindent Fact (2):  if $T$ is an algebra and $A=T^{(c)}$,  then for every ideal  $I\subset A$ and $f\in A$ one has $IT\cap A=I$  and $(IT:_Tf)\cap A=I:_Af$.
\medskip 

The first is an elementary and well-know property of monomials, the second holds true because $A$ is a direct summand of $T$. 

Let $A=R^{(c)}$. Let $X$ be the set of the residue classes in $R$ of the monomials of degree $c$ that are not in $I$. Clearly $X$ is a basis of $A_1$. Let $Y\subset X$ and $z\in X\setminus Y$, say $Y=\{\bar m_1, \dots, \bar m_v\}$ and $z=\bar m$. We have to compute $(Y):_Az$. To this end  let us consider $J=(I+(m_1, \dots, m_v)):_S m$ and note that $J=I+H$ with $H$ a monomial ideal generated in degrees $\leq c$. Then 
$(Y):_Az=( \bar m : m\in H\setminus I \mbox{ is a monomial of degree } c)$. 
\end{proof} 
 Let us single out two interesting special cases: 
\begin{theorem}
Let $S=K[x_1,\dots,x_n]$. Then $S^{(c)}$ is strongly Koszul for every $c\in \NN$. 
 \end{theorem} 
\begin{theorem}
\label{monKos}
Let $S=K[x_1,\dots,x_n]$ and let $I\subset S$ be an ideal generated by monomials of degree $2$. Then $S/I$ is strongly Koszul. 
 \end{theorem} 
 
Given a Koszul filtration $\F$ for an algebra $R$ we may also look at modules having a filtration compatible with $\F$. 
This leads us to the following: 

\begin{definition} Let $R$ be an algebra  with a Koszul filtration $\F$. Let $M$ be an $R$-module. We say that $M$ has linear quotients with respect to $\F$ if $M$  is minimally generated by elements $m_1,\dots,m_v$ such that $\langle m_1,\dots, m_{i-1} \rangle:_Rm_i\in \F$ for $i=1,\dots,v$. 
\end{definition}

One easily deduces: 

\begin{lemma}
\label{facile}
 Let $R$ be an algebra with a Koszul filtration $\F$ and  $M$ an $R$-module with linear quotients with respect to $\F$. Then $\reg_R (M)=t_0^R(M)$. 
\end{lemma}

As an example we have: 

\begin{proposition}
\label{verofil}
Let $S=K[x_1,\dots,x_n]$ and $I$ be a monomial ideal generated in degree $\leq d$. Consider $R=S/I$ and the Veronese ring $R^{(c)}$ equipped with the Koszul filtration described in the proof of \ref{sKos}. For every $u=0,\dots,c-1$ the Veronese module $V_u=\oplus_j R_{u+jc}$ has   linear quotients with respect to $\F$.  
\end{proposition} 

The proof is easy, again based on Fact (1) in the proof of \ref{sKos}.   In particular, this gives another proof of the fact that the Veronese modules $V_u$ have  a linear $R^{(c)}$-resolution.

The results and the proofs presented for Veronese rings and Veronese modules have their analogous in the multigraded setting, see \cite{CHTV}. For later applications we mention explicitly one case. 

Let $S=K[x_1,\dots,x_n,y_1,\dots,y_m]$ with $\ZZ^2$-grading induced by the assignment $\deg(x_i)=(1,0)$ and $\deg(y_i)=(0,1)$. For every $c=(c_1,c_2)$ we look at the diagonal subalgebra $S_\Delta=\oplus_{a\in \Delta} S_{a}$ where $\Delta=\{ic : i\in \ZZ\}$. The algebra $S_\Delta$ is nothing but the Segre product of the $c_1$-th Veronese ring of $K[x_1,\dots,x_n]$ and the $c_2$-th Veronese ring of $K[y_1,\dots,y_m]$. We have: 

\begin{proposition}
\label{Segremod}
For every $(a,b)\in \ZZ^2$ the  $S_\Delta$-submodule of $S$ generated by $S_{(a,b)}$ has a linear resolution. 
 \end{proposition}

\section{Absolutely and universally}
\label{stronKo}
We have discussed  in the previous sections some notions,  such as being G-quadratic, strongly Koszul, having a Koszul filtration or  a Gr\"obner flag that imply Koszulness. In the present section we discuss two variants of the Koszul property: universally Koszul and absolutely Koszul. 

\subsection{Universally Koszul} 
\

When looking for a Koszul filtration, among the many families of ideals of linear forms  one can take the set of all ideals of linear forms. This leads to the following definition: 

\begin{definition} Let $R$ be a $K$-algebra and set 
$$\LL(R)=\{ I\subset R : I \mbox{ ideal generated by elements of degree } 1\}. $$
We say that $R$ is universally Koszul if the following equivalent conditions hold: 
\begin{itemize} 
\item[(1)] $\LL(R)$  is a Koszul filtration of $R$.  
\item[(2)] $\reg_R (R/I)=0$ for every $I\in \LL(R)$.  
\item[(3)] For every $I\in \LL(R)$ and $x\in R_1\setminus I$ one has $I:x\in \LL(R)$.
\end{itemize} 
\end{definition} 

That the three conditions are indeed equivalent is easy to see, see \cite[1.4]{C1}. In \cite[2.4]{C1} it is proved that:

 \begin{theorem}
\label{uk1} 
Let $S=K[x_1,\dots,x_n]$ and $m\in \NN$. If  $m\leq n/2$, then a generic space of quadrics  of  codimension $m$ in the vector space of quadrics defines a  universally Koszul algebra.    
\end{theorem} 

One should compare the result above with the following:

\begin{theorem}
\label{generic} 
Let $S=K[x_1,\dots,x_n]$ and $m\in \NN$.  
\begin{itemize}
\item[(1)] A generic space of quadrics of codimension $m$ defines a Koszul algebra if $m\leq n^2/4$.
\item[(2)] A generic space of quadrics of codimension $m$ defines an  algebra with a  Gr\"obner flag if $m\leq n-1$. 
\end{itemize} 
\end{theorem} 

For (1) see \cite[3.4]{CTV}, for (2) \cite[10]{C3}.  Fr\"oberg and L\"ofwall proved  in \cite{FL}  that, apart from spaces of quadrics of codimension $\leq n^2/4$, the only generic spaces of quadrics defining Koszul algebras are the complete intersections.   
Returning to universally Koszul algebras,  in \cite{C1} it is also proved that:

\begin{theorem}
\label{uk2} 
Let $R$ be a Cohen-Macaulay domain $K$-algebra with $K$ algebraically closed of characteristic $0$. 
Then $R$ is universally Koszul if and only if $R$ is a polynomial extension of one of the following algebras: 
\begin{itemize} 
\item[(1)] The coordinate ring of a quadric hypersurfice.
\item[(2)] The coordinate ring of a rational normal curve, i.e. $K[x,y]^{(c)}$ for some $c$.
\item[(3)] The coordinate ring of a rational normal scroll of type $(c,c)$, i.e.   the Segre product of $K[x,y]^{(c)}$ with $K[s,t]$. 
\item[(4)] The coordinate ring of  the Veronese surface in $\PP^5$, i.e. $K[x,y,z]^{(2)}$. 
\end{itemize} 
\end{theorem}

 \subsection{Absolutely Koszul}\label{AK}

  Let $\FF^R_M$ be the minimal free resolution of a graded module $M$ over $R$. One defines a $\mm_R$-filtration on $\FF^R_M$ whose associated graded complex $\lin(\FF^R_M)$ has, in the graded case,  a very elementary description. The complex  $\lin(\FF^R_M)$ is obtained from $\FF_M$ by replacing with $0$  all entries of degree $>1$  in the matrices representing the homomorphisms.  In the local case the definition of $\lin(\FF^R_M)$ is more complicated,  see Section \ref{locca}  for details.  One defines 
 \begin{equation} 
 \label{ldgr} 
 \ld_R(M)=\sup \{ i: H_i(\lin (\FF^R_M))\neq 0\}.
 \end{equation}  
 Denote by $\Omega^R_i(M)$ the $i$-th syzygy module of a module $M$ over $R$. 
  It is proved in  R\"omer  PhD thesis and also in \cite{IR} that: 
 
 \begin{proposition} 
\label{romer}
 Assume $R$ is Koszul. Then: 
 \begin{itemize}
 \item[(1)]  $M$ is componentwise linear iff $\ld_R(M)=0$.
 \item[(2)]  $\ld_R(M)=\inf\{ i: \Omega_i(M) \mbox{ is componentwise linear} \}$. 
 \item[(3)] If $\Omega^R_i(M)$ is componentwise linear then $\Omega^R_{i+1}(M)$ is componentwise linear.
 \end{itemize} 
 \end{proposition} 
 
 Iyengar and R\"omer introduced in \cite{IR} the following notion: 
 
\begin{definition} 
\label{abkos} 
A $K$-algebra $R$  is said to be absolutely Koszul if  $\ld_R(M)$ is finite for every module $M$. 
\end{definition}

It is shown in \cite{HI} that: 

\begin{proposition} If $\ld_R(M)$ is finite, then $\reg_R(M)$ is finite as well. Furthermore  the Poincar\'e series $P_M(z)$ of $M$ is rational and its ``denominator" only depends on $R$. 
\end{proposition} 

One obtains the following characterization of the Koszul property: 

\begin{corollary}
\label{finite ld(k)} Let $R$ be a $K$-algebra. Then  $R$ is Koszul if and only if $\ld_R(K)$ is finite. In particular, if $R$ is absolutely Koszul then $R$ is Koszul. \end{corollary} 
 
On the other hand there are Koszul algebras  that are not absolutely Koszul. 

\begin{example} The algebra  $$R=K[x_1,x_2,x_3,y_1,y_2,y_3]/(x_1 , x_2 , x_3)^2+(y_1 , y_2 , y_3)^2$$ is Koszul  but not absolutely Koszul because there are $R$-modules with non-rational Poincar\'e series. This and other examples of  ``bad" Koszul algebras are discussed by Roos in \cite{R2}.
\end{example}

One also has \cite[5.10]{HI}
\begin{theorem} 
\label{absKos}
Let $R=S/I$ with $S=K[x_1,\dots,x_n]$. Then $R$ is absolutely Koszul if either $R$ is   a complete intersection of quadrics or $\reg_S(R)=1$. 
\end{theorem} 
There is however an important difference between the two cases  \cite[6.2, 6.7]{HI}: 
\begin{remark} 
\label{absKos1}
When $\reg_S(R)=1$ one has $\ld_R(M)\leq 2\dim R$ for every $M$ and even $\ld_R(M)\leq \dim R$ is furthermore $R$ is Cohen-Macaulay.   But when $R$ if a complete intersection of quadrics  of codimension $>1$ (or more generally when $R$ is Gorenstein of with socle in degree $>1$) one has $\sup_M  \ld_R(M)=\infty$.\end{remark} 

Another important contribution is the following:  
\begin{theorem}\label{1nn1} Let $R$ be a Gorenstein Artinian algebra with Hilbert function $1+nz+nz^2+z^3$ and  $n>2$. Then 
\begin{itemize} 
\item[(1)] If  there exist $x,y\in R_1$ such that $0:x=(y)$  and $0:y=(x)$  (an exact pair of zero-divisors in the terminology of \cite{HS}) then $R$ has a Koszul filtration and it is absolutely Koszul.  
\item[(2)] If $R$ is generic then it has an exact pair of zero-divisors.
\end{itemize} 
\end{theorem} 
See \cite[2.13,6.3]{CRV} for the statement on Koszul filtration in (1) and for (2) and see \cite[3.3]{HS} for the the absolutely Koszulness. 

What are the relationships between the properties  discussed in this and the  earlier sections? Here are some questions: 
 \begin{question}
\label{q8} 
\begin{itemize} 
\item[(1)] Strongly Koszul $\Rightarrow$ G-quadratic?
\item[(2)] Universally  Koszul $\Rightarrow$ G-quadratic?
\item[(3)] Universally  Koszul $\Rightarrow$ absolutely Koszul?
\end{itemize} 
\end{question}
Question \ref{q8} (1) is mentioned in \cite[pg.166]{HHR} in the toric setting.  
 Another interesting question is:  

\begin{question}
\label{q9} 
What is the behavior of  absolutely Koszul algebras under standard algebra operations  (e.g.~forming Veronese subalgebras or Segre and fiber  products)? 
\end{question}
The same question for universally Koszul algebras is discussed in \cite{C1} and for strongly Koszul in \cite{HHR}. Note however that in  \cite{HHR} 
 the authors deal mainly with toric algebras and their toric coordinates.  Universally Koszul algebras with monomial relations have been classified in \cite{C2}. We may ask
 \begin{question}
\label{q10} 
Is it possible to classify absolutely Koszul algebras defined by monomials? 
\end{question}

\section{Regularity and Koszulness}\label{QuePro}

We list in this section some facts and some questions that we like concerning Koszul algebras and regularity. 
 We observe the following. 
\begin{remark} 
\label{IMS}
Regularity over the polynomial ring $S$ behaves quite well with respect to products of ideals and modules. 
\begin{itemize} 
\item[(1)] $\reg_S(I^uM)$ is a linear function in $u$ for large $u$, see \cite{CHT,K,TW}. 
\item[(2)] $\reg_S(IM)\leq \reg_S(M)+\reg_S(I)$ (does not hold in general but it) holds provided $\dim S/I\leq 1$, \cite{CH}. 
\item[(3)] More generally, $$\reg_S (\Tor_i^S(N,M))\leq \reg_S (M)+\reg_S (N)+i$$ provided the Krull dimension of  $\Tor_1^S(N,M)$ is $\leq 1$, \cite{Ca1,EHU}. 
\item[(4)] $\reg_S(I_1\cdots I_d)=d$ for ideals $I_i$ generated in degree $1$, \cite{CH}
\end{itemize} 
where $M,N$ are $S$-modules and $I,I_i$ are ideals of $S$. 
\end{remark}
What happens if we replace in \ref{IMS} the polynomial ring $S$ with a Koszul algebra $R$ and  consider regularity over $R$? 
 Trung and Wang proved in \cite{TW} that  $\reg_S(I^uM)$ is asymptotically a linear function in $u$ when $I$ is an ideal of  $R$ and $M$ is a $R$-module. If $R$ is Koszul,  $\reg_R(I^uM)\leq \reg_S(I^uM)$ and hence $\reg_R(I^uM)$ is bounded above by a linear function in $u$. 

\begin{question}\label{q11} Let $R$ be a Koszul algebra  $I\subset R$ an ideal and $M$ an $R$-module.  Is $\reg_R(I^uM)$ a linear function in $u$ for large $u$? 
\end{question} 

The following examples show that statements (2) and (3)   in   \ref{IMS}  do not hold over Koszul algebras. 

\begin{example}\label{exIMS} Let $R=\QQ[x,y,z,t]/(x^2+y^2, z^2+t^2)$. With $I=(x,z)$ and $J=(y,t)$ one has $\reg_R(I)=1$, $\reg_R(J)=1$ because $x,z$ and $y,t$ are regular sequences on $R$, $\dim R/I=0$  and $\reg_R(IJ)=3$.
\end{example} 

\begin{example} Let $R=K[x,y]/(x^2+y^2)$.  Let  $M=R/(x)$ and $N=R/(y)$ and note that $\reg_R(M)=0$, $\reg_R(N)=0$ because $x$ and $y$ are non-zero divisors in $R$ while $\Tor_1^R(M,N)=H_1(x,y,R)=K(-2)$. 
\end{example}
 
Nevertheless statements (2),(3) of \ref{IMS} might hold for special type of ideals/modules over  special type of Koszul algebras. For example one has: 

\begin{proposition} 
\label{scarti1}
Let $R$ be a Cohen-Macaulay $K$-algebra with $\reg_S(R)=1$, let $I$ be an ideal generated in degree $1$ such that $\dim R/I\leq 1$ and $M$ an $R$-module.  Then  $\reg_R(IM)\leq \reg_R(M)+1$. In particular, $\reg_R(I)=1$. 
\end{proposition} 
\begin{proof} We may assume $K$ is infinite.  The short exact sequence  
$$0\to IM\to M\to M/IM\to 0$$ 
implies that $\reg_R(IM)\leq \max\{\reg_R(M), \reg_R(M/IM)+1\}$. It is therefore enough to prove that $\reg_R(M/IM)\leq \reg_R(M)$. Then let $J\subset I$ be an ideal generated by a maximal regular sequence of elements of degree $1$ and set $A=R/J$. Since $\reg_R(A)=0$ and since $M/IM$ is an  $A$-module, by virtue of \ref{Giux2}, we have $\reg_R(M/IM)\leq \reg_A(M/IM)$. By construction, $A$ is Cohen-Macaulay of dimension $\leq 1$ and has regularity $1$ over the polynomial ring projecting onto it. So, by \ref{a1} we have  $\reg_A(M/IM)=\max\{ t_0^A(M/IM), t_1^A(M/IM)-1\}$. Summing up, since $t_0^A(M/IM)=t_0^R(M)$, it is enough to prove  $t_1^A(M/IM)\leq \reg_R(M)+1$.  Now we look at
$$0\to IM/JM\to M/JM\to M/IM\to 0$$ 
that gives $t_1^A(M/IM)\leq \max\{ t_1^A(M/JM), t_0^A(IM/JM)\}$. Being $t_0^A(IM/JM)\leq t_0^R(M)+1\leq \reg_R(M)+1$, it remains to prove that 
$t_1^A(M/JM)\leq  \reg_R(M)+1$, and this follows from the right exactness of the tensor product. 
\end{proof} 

The following example shows that the assumption $\dim R/I\leq 1$ in \ref{scarti1} is essential. 

\begin{example}  The algebra $R=K[x,y,z,t]/(xy,yz,zt)$ is Cohen-Macaulay of dimension $2$ and $\reg_S(R)=1$. The ideal $I=(y-z)$ has $\reg_R(I)=2$ and $\dim R/I=2$. 
\end{example} 

Example \ref{exIMS} shows also that statement (4) of \ref{IMS} does not hold over a Koszul algebra even if  we assume that each $I_i$ is an ideal of regularity $1$ and of finite projective dimension.  Statement (4) of \ref{IMS}  might be true if one assumes that the ideals $I_i$ belongs to a Koszul filtration. We give a couple of examples in this direction: 

\begin{proposition} 
\label{prodstrKo}
Let $S=K[x_1,\dots,x_n]$, $R=S/I$ with $I$ generated by monomials of degree $2$. Let $X=\{\bar x_1, \dots,\bar x_n\}$ and $\F=\{ (Y) : Y\subset X\}$. Let $I_1,\dots,I_d\in \F$. Then $\reg_R(I_1\cdots I_d)=d$ unless $I_1\cdots I_d=0$. 
\end{proposition} 
\begin{proof} First we observe the following. Let  $m_1,\dots,m_t$ be  monomials of degree $d$ and $J=(m_1,\dots,m_t)$. Assume that they have linear quotients (in $S$), that is $(m_1,\dots, m_{i-1}):_Sm_i$ is generated by variables for every $i$. Fact (1) in the  proof of \ref{sKos}  implies  that $JR$ has  linear quotients with respect to the Koszul filtration $\F$ of $R$.  By \ref{facile} we have that  $\reg_R(JR)=d$ (unless $JR=0$).  Now the desired result follows because products of ideals of variables  have linear quotients in $S$ by  \cite[5.4]{CH}.  \end{proof} 

Example 4.3 in \cite{CH} shows that the inequality $\reg_R(IM)\leq \reg_R(M)+\reg_R(I)$ does not even hold over a $K$-algebra  $R$  with a Koszul filtration $\F$, $I\in \F$ and $M$ an $R$-module  with  linear quotient with respect to $\F$. 
The following are natural questions: 

\begin{question}\label{q12}
Let $R$ be an algebra with a Koszul filtration $\F$. Is it true that  $\reg_R(I_1\cdots I_d)=d$ for every $I_1,\dots,I_d\in \F$ whenever the product is non-zero?  
\end{question}

In view of the analogy with statement (4) of \ref{IMS} the following special case deserves attention: 

\begin{question}\label{q12bis}
Let $R$ be a universally Koszul  algebra.  Is it true that  $\reg_R(I_1\cdots I_d)=d$ for every $I_1,\dots,I_d$ ideals of $R$  generated in degree $1$ (whenever the product is non-zero)?  
\end{question}

\begin{remark} 
\label{propri}
In a universally Koszul algebra a product of elements of degree $1$ has a linear annihilator. 
This can be easily shown by induction on the number of factors.  Hence the answer to \ref{q12bis} is positive if each $I_i$ is principal. 
\end{remark}

We are able to answer \ref{q12} in the following cases: 

\begin{theorem}
\label{scarti2}
Products of ideals of linear forms have  linear resolutions over the following rings: 
\begin{itemize}
\item[(1)] $R$ is Cohen-Macaulay with $\dim R\leq 1$ and $\reg_S(R)=1$.
\item[(2)] $R=K[x,y,z]/(q)$ with $\deg q=2$.
\item[(3)] $R=K[x,y]^{(c)}$ with $c\in \NN_{>0}$. 
\item[(4)] $R=K[x,y,z]^{(2)}$.
\item[(5)] $R=K[x,y]*K[s,t]$ ($*$ denotes the Segre product). 
\end{itemize} 
\end{theorem} 

\begin{proof}  The rings in the list are Cohen-Macaulay with $\reg_S(R)=1$. Let $I_1,\dots,I_d$ be ideals generated by linear forms. We prove by induction on $d$ that $\reg_R(I_1\cdots I_d)=d$. The case $d=1$ follows because the rings in the list are universally Koszul. If for one of the $I_i$ we have  $\dim R/I_i\leq 1$ then we may use \ref{scarti1} and conclude by induction. Hence we may assume $\dim R/I_i\geq 2$ for every $i$. For the ring (1) and  (3) (which is a $2$-dimensional domain) we are done.  In the case (2),  the only case left is when the  $I_i$ are  principal.  But then we may conclude by virtue of \ref{propri}.  In cases  (3) and (4) we have that  $\dim R/I_i=2$ for each $i$, that is $\height I_i=1$. If one of the $I_i$ is principal, then we are done by induction (because the $R$ is a domain). Denote by $A$ either  $K[x,y,z]$ in case (3) or $K[x,y,s,t]$ in case (4). Since $R$ is a direct summand of $A$ we have $IA\cap R=I$ for every  ideal $I$ of $R$. It follows that $\height (I_iA)=1$ for every $i$ and hence there exist non-units $f_i\in A$ such that $I_iA\subset (f_i)$. In case (3) we have that each $f_i$ must have degree $1$ in $A$ and $I_iA=(f_1)J_1$ with $J_1$ an ideal generated by linear forms of $A$. Hence $I_1I_2=(f_1f_2)H$ where $H=J_1J_2$ is an ideal generated by linear forms of $R$.  Hence we are done because one of the factor is principal. In case (4) we have that each $f_i$ is either a linear form in $x,y$ or a linear form in $s,t$.  If one of the $f_i$'s  is a linear form in  $x,y$ and  another one  is a linear form in  $s,t$ we can proceed as in the case (3). So we are left with the case that every $f_i$ is a linear form in, say, $x$ and $y$ and $I_i=(f_i)J_i$ with $J_i$ generated by linear forms in $z,t$. Since none of the $I_i$ is principal we have that $J_i=(z,t)$ for every $i$. Hence $I_1\cdots I_d$ is generated by $(\prod_{i=1}^d  f_i)(z,t)^d$ and it  isomorphic  to the $R$-submodule of $A$ generated by its component of  degree  $(0,d)$. That such a module has a linear resolution over $R$ follows from  \ref{Segremod}. \end{proof}

 We state now a very basic question of computational nature. 
 
\begin{question}
\label{q1} 
Let $R$ be a Koszul algebra and $M$ an $R$-module.  How does one compute $\reg_R(M)$? Can one do it algorithmically? \end{question}

Few comments concerning Question \ref{q1}. We assume to be able to compute syzygies over $R$ and so to be able to compute the first steps of the resolution of a $R$-module $M$.  Let $S\to R$ the canonical presentation (\ref{canpre}) of $R$. We know that 
$\reg_R(M)\leq \reg_S(M)$ and $\reg_S(M)$ can be computed algorithmically because $\projdim_S(M)$ is finite. 
A special but already interesting case of \ref{q1} is: 
\begin{question}
\label{q2} 
Let $R$ be a Koszul algebra and $M$ an $R$-module generated in degree $0$, with $M_1\neq 0$ and $M_i=0$ for $i>1$.  Can one decide algorithmically whether $\reg_R (M)=0$ or $\reg_R (M)=1$?  
 \end{question}

Set 
$$r_R(M)=\min\{ i\in\NN : t_i^R(M)-i=\reg_R(M)\}.$$
So $r_R(M)$ is the first homological position where the regularity of $M$ is attained. 
If one knows $r_R(M)$ or a upper bound $r\geq r_R(M)$ for it,  then one can compute $\reg_R(M)$ by computing the first $r$ steps of the resolution of $M$. Note that 
$$r_R(M)\leq \ld_R(M)$$ 
because $\reg_R(N)=t_0^R(N)$ if $N$ is componentwise linear. One has: 

\begin{lemma}
\label{a1} 
Let $R$ be a $K$-algebra with $\reg_S(R)=1$.  Then $r_R(M)\leq 2\dim R$ for every $M$, i.e. the regularity of any $R$-module is attained within  the first $2\dim R$ steps of the resolution. If furthermore $R$ is Cohen-Macaulay,  $r_R(M)\leq \dim R-\depth M$. 
\end{lemma}
The first assertion follows from \ref{absKos1}, the second is proved by a simple induction on $\depth M$.

Note that the $i$-th syzygy module of $M$ cannot have a free summand if $i>\dim R$ by \cite[0.1]{E} and so 
$$t_{j+1}^R(M)>t_j^R(M) \mbox{ if } j>\dim R.$$ 

Unfortunately there is no hope to get a  bound for $r_R(M)$ just in terms of invariants of $R$ for general Koszul algebras. The argument of \cite[6.7]{HI} that shows that  if $R$ is a Gorenstein algebra with socle degree $>1$ then $\sup \ld_R(M)=\infty$ shows also  that  $\sup_M r_R(M)=\infty$. For instance, over $R=K[x,y]/(x^2,y^2)$ let $M_n$ be the dual  of the $n$-th syzygy module $\Omega^R_n(K)$ shifted by $n$. One has that $M_n$ is generated be in degree $0$,  $\reg_R (M_n)=1$ and $r_R(M_n)=n$. On the other hand,  the number of generators of $M_n$ is $n$.  So we ask: 

\begin{question}
\label{q3} 
Let $R$ be a  Koszul algebra. Can one bound $r_R(M)$ in terms of invariants of $R$ and ``computable" invariants of $M$ such as its Hilbert 
series or  its Betti numbers over $S$?  \end{question}

The questions above make sense also over special families of Koszul rings. For instance,  there has been a lot of activity to understand resolutions of modules over short rings, i.e.~rings  with $\mm_R^3=0$ or $\mm_R^4=0$,  both in the graded and local case, see \cite{AIS1,AIS2,HS}. It would be very interesting to answer Questions \ref{q1}, \ref{q2} and \ref{q3}  for  short rings.

 \section{Local variants}
 \label{locca}

 This section is concerned with ``Koszul-like"  behaviors of  local rings and their modules.   
  
\medskip
\noindent {\bf Assumption}: From now on, when not explicitly said, $R$ is assumed to be a local or graded ring with  maximal ideal $\mm$ and residue field $K=R/\mm$. Moreover all modules and ideals are  finitely generated, and homogeneous in the graded case. 
 \medskip 

 We define the associated graded ring to $R$ with respect to the $\mm$-adic filtration 
 $$G=\gr_{\mm}(R)=\oplus_{i\ge0} \mm^i/\mm^{i+1} .$$ 
The Hilbert series and the Poincar\'e series of $R$ are:

$$ H_R(z)=H_G(z)=\sum_{i\ge 0} \dim (\mm^i/\mm^{i+1})  z^i\ \ \   \text{ and}\ \ \ P_R(z)=\sum_{i\ge 0} \dim \Tor_i^R(K,K) z^i.$$
 
\subsection{Koszul rings}

Following  Fr\"oberg \cite{Fr} we extend the definition of Koszul ring to the local case as follows:
 
\begin{definition}
\label{Kl} The ring $R$ is Koszul if its associated graded ring $G$ is a Koszul algebra (in the graded sense), that is, $R$ is Koszul if $K$ has a linear resolution as a $G$-module. 
\end{definition}

As it is said in Remark \ref{R16} (6) in the graded setting  the Koszul property holds  equivalent to the following  relation between the Poincar\'e series of $K$ and the Hilbert series of $R$:     
\begin{equation} 
\label{PH}  
P^R_K(z) H_R(-z)=1.
\end{equation}

The following definition is due to Fitzgerald \cite{Fit}: 
\begin{definition}
The ring $R$ is Fr\"oberg if  the relation (\ref{PH}) is verified.
\end{definition}
 
We want to explain why every Koszul ring is Fr\"oberg. To this end we need to introduce few definitions.

Let  $A$ be a regular local ring  with maximal ideal $\mm_A$ and let $I$ be an ideal of $A$   such that $I\subseteq \mm^2_A$. Set $R=A/I$.  Then $G \simeq S/I^* $ where $S$ is the polynomial ring and $I^*$ is the homogeneous ideal generated by the initial forms $f^*$ of the elements $f \in I$. 

\begin{definition}
\label{isomultiple}
\begin{itemize}
\item[(1)] A subset $\{f_1, \dots, f_t\}$ of $I$ is a standard basis of $I$ if $I^*=(f_1^*, \dots, f_t^*); $

\item[(2)] The ideal $I$ is $d$-isomultiple if  $I^*$ is generated in degree $d.$ 
\end{itemize}
\end{definition}

If $\{f_1, \dots, f_t\}$ is a standard basis of $I$ then $I=(f_1, \dots, f_t).$ 
See \cite{RV1} for more details on $d$-isomultiple ideals. 
Notice that  by Remark \ref{R16} (1) we have 
$$R\ \ {\text{  Koszul}} \ \ \implies \ \  I \ {\text{is}}  \    2{\text{-isomultiple}}.$$  Obviously  the converse does not hold true because a quadratic $K$-algebra is not necessarily Koszul. 
  
We now explore the connection between Koszul and Fr\"oberg rings.
By definition $H_R(z)=H_G(z),  $ and    
$$P^R_K(z) \le P^G_K(z)   $$
(see for example  \cite[4]{Fr1}).   Conditions are known under which $\beta_i^R(K)=\beta_i^G(K),$ for instance this happens  if  
\begin{equation}
\label{frofro}
t_i^G(K)=\max \{ j : \beta_{ij}^G(K)\neq 0\} \le  \min \{ j : \beta_{i+1 j}^G(K)\neq 0\}  \text{ for every }i,
\end{equation} 
see  \cite[4]{Fr1}. This is the case  if $K$ has a linear resolution as a $G$-module. Hence

\begin{proposition}
\label{KF}
 If  $R $ is Koszul, then $R$ is Fr\"oberg.
\end{proposition}
\begin{proof}  By definition, $H_R(z)=H_G(z)$. If $R$ is a Koszul ring, then $G$ is Koszul, in particular $P^G_K(z)H_G(-z)=1.$ The result follows because the graded  resolution of $K$ as a $G$-module is linear and hence (\ref{frofro}) and therefore $P^R_K(z)=P^G_K(z)$.   
\end{proof} 

Since in the graded case $R$ is Fr\"oberg iff it  is  Koszul,  it is natural to ask the following question.

\begin{question} 
\label{FK}
Is a Fr\"oberg (local) ring Koszul? 
\end{question}

We  give a positive answer to this question for a special class of rings. If $f $ is a non-zero element of R, denote by $v(f)=v $ the valuation of $f, $ that is the largest integer such that $f \in \mm^v. $

\begin{proposition} Let $I$ be an ideal generated by a regular sequence in a regular local ring $A$. The following facts are equivalent: 
 \begin{itemize}
 \item[(1)]  $A/I$ is Koszul.
 \item[(2)]   $A/I$ is Fr\"oberg.
 \item[(3)]  $I$ is $2$-isomultiple.
 \end{itemize} 
 \end{proposition} 
 \begin{proof} By  \ref{KF} we know (1) implies (2). We prove that (2) implies (3). Let $I=(f_1,\dots, f_r) $ with $v(f_i)=v_i \ge 2. $  By  \cite{T} we have  $P^{A/I}_K(z)=(1+z)^n/  (1-z^{2})^r$.   Since $A/I$ is Fr\"oberg, one has that $H_{A/I} (z)=(1-z^{2})^r/(1-z)^n,$ in particular the multiplicity of $A/I$ is $2^r$.  From \cite[1.8]{RV1}, it follows that $v_i=2 $ for every $i=1,\dots, r$ and $f_1^*, \dots, f_r^* $ is a regular sequence in $G$.  Hence $I^*=(f_1^*, \dots, f_r^*),   $  so $I$ is $2$-isomultiple. If we assume (3), then $G$ is a graded quadratic complete intersection, hence  $P^G_K(z) H_G(-z)=1$ and since $G$ is graded this implies that $G$ is  Koszul.  
 \end{proof} 
  
Next example is  interesting to better understand what happens in case the regular sequence is not  $2$-isomultiple. 
  
\begin{example} Consider $I_s=(x^2-y^s, xy) \subset A=K[[x,y]] $ where $s$ is an integer $\ge 2.$  Then,  as we have seen before,    $P^{A/I}_K(z)=(1+z)^2/(1-z^2)^2 $ and it does not depend on $s.$ On the contrary  the Hilbert series depends on $s, $ precisely $H_{A/I}(z)=1 +2z+ \sum_{i=2}^s z^i.$ It follows that  $A/I$ is Koszul  (hence Fr\"oberg)  if and only if $s=2 $ if and only if $I$ is $2$-isomultiple. In fact if $s>2, $ then $I_s^*=(x^2,xy, y^{s+1}) $ is not quadratic. 
\end{example}

In the following we denote by $e(M)$ the multiplicity (or degree) of  an $R$-module $M$ and by $\mu(M)$ its minimal number of generators. Let $R$ be a Cohen-Macaulay ring.  Abhyankar proved that  $e(R)\ge h+1 $  and $h=\mu(\mm)-\dim R$ is  the so-called embedding codimension. If equality holds  $R$ is said to be of minimal multiplicity.

\begin{proposition} 
\label{h+12}
Let $R$ be a  Cohen-Macaulay  ring of multiplicity $e$ and Cohen-Macaulay type $\tau$.  If one of the following conditions holds: 
 \begin{itemize}
 \item[(1)]  $e=h+1$,
 \item[(2)] $e=h +2 $ and $\tau <h,$
 \end{itemize} 
then $R$ is a Koszul ring.

\end{proposition}
\begin{proof} In both cases the associated graded ring is Cohen-Macaulay and quadratic (see  \cite[3.3, 3.10]{RV1}). We may assume that the residue field is infinite, hence there exist $x_1^*, \dots, x_d^*$ filter regular sequence in $G$ and it is enough to  prove that $G/(x_1^*, \dots, x_d^*) \simeq \gr_{\mm/(x_1, \dots, x_d)}  (R/(x_1, \dots, x_d))  $ is Koszul (see for example \cite[2.13]{IR}). Hence the problem is reduced to an Artinian quadratic $K$-algebra  with $\mu(\mm)=h  >1 $ and $\dim_K\mm^2 \le 1  $  and the result follows, see   \cite{Fit}  or \cite{C3}. 
\end{proof} 
 
\begin{remark} (1) There are Cohen-Macaulay rings $R$ with  $e=h +2   $ and $\tau=h $ whose associated graded ring $G$ is not quadratic, hence not Koszul.  For example this is the case if  $R=k[[t^5, t^6, t^{13}, t^{14}]], $ where  $e=h+2=3+2=5 $ and $\tau=3.  $  

(2) Let  $R$ be Artinian of multiplicity $e=h +3$.  Then $R$ is stretched  if its Hilbert function is $1+hz+z^2+z^3,$  and short if its Hilbert function is $1+hz+h^2$  (for details see  \cite{RV2}). For example if $R$ is Gorenstein, then $R$ is stretched. Sally classified,  up to analytic isomorphism,  the Artinian local rings which are stretched in terms of the multiplicity  and the Cohen-Macaulay type. As a consequence one verifies  that if  $R$  is stretched of multiplicity $\ge h+3, $  then $I^*$ is never quadratic, hence $R$ is never Koszul.  If $R$ is short,    then $R$ is Koszul if and only if $G$ is quadratic.  In fact,  by a result in  Backelin's PhD thesis  (see also \cite{C3}),  if  $\dim_K G_2=2 $ and $G$ is quadratic, then $G$ is Koszul, so $R$ is Koszul.  
\end{remark} 
  
\subsection{Koszul modules and linear defect }\label{modules}
  
Koszul modules have been introduced in \cite{HI}. Let us recall the definition.   
   
Consider $(\FF^R_M, \delta) $ a minimal free resolution of $M$ as an $R$-module. The property $ \delta (\FF^R_M) \subseteq \mm \FF^R_M $ (the minimality) allows us to form for every $j \ge 0$ a complex  
 $$   { \lin_j}(\FF^R_M) :   \ \    0 \to \frac {F_j }{\mm F_j} \to \dots \to \frac{ \mm^{j-i} F_i}{ \mm^{j-i+1} F_i} \to \dots \to \frac{\mm^jF_0}{\mm^{j+1} F_0}  \to 0$$ 
of $K$-vector spaces.  
Denoting $\lin(\FF^R_M)=\oplus_{j\ge 0} \lin_j (\FF^R_M)$, one has that $\lin (\FF^R_M) $ is a complex of free graded modules over $G=\gr_{\mm}(R) $ whose  $i$-th free module  is $$\oplus_j  \mm^{j-i} F_i/\mm^{j-i+1} F_i=\gr_{\mm}(F_i)(-i)=G(-i) \otimes_K F_i/\mm F_i. $$ By  construction the differentials can be described by matrices of linear forms. 
\vskip 2mm
 Accordingly with  the definition given  by Herzog and Iyengar in \cite{HI}: 
\begin{definition} \label{Kos} $M$ is a Koszul module if $H_i(  \lin_j (\FF^R_M))=0 $   for every $i > 0 $ and $j\ge 0$, that is,  $H_i(  \lin(\FF^R_M))=0 $ for every $i > 0$.
\end{definition}

\begin{remark}
Notice  that, if $R$ is  graded, the $K$-algebras $G$ and $ R$ are naturally isomorphic. In particular  $\lin(\FF^R_M)$ coincides  with the complex  already defined in \ref{AK}. This is why $\lin(\FF^R_M)$ is called the  linear part of $\FF^R_M$. 
\end{remark}

  As in the graded case, see (\ref{ldgr}),  one defines the linear defect of $M$ over $R$: 
 \begin{equation}   \ld_R(M)=\sup \{ i: H_i(\lin (\FF^R_M))\neq 0\}.\end{equation} 
The linear defect  gives a measure of how far is $ \lin(\FF^R_M)$ from being a resolution of  $\gr_{\mm}(M)=  \oplus_{j\ge 0} \mm^j M/\mm^{j+1} M.$ By the uniqueness of minimal free resolution, up to isomorphism of complexes, one has that $\ld_R(M) $ does not depend on $\FF^R_ M, $ but only on the module $M.$
When  $\ld_R(M)< \infty  $ we say that  in the minimal free resolution the linear part predominates.

Koszul modules have appeared previously in the literature under the name ``modules with linear resolution" in \cite{Se} and 
 ``weakly Koszul"  in \cite{MVZ}.

By definition  $R$ is Koszul as an $R$-module because it is $R$-free.  But, accordingly with  Definition \ref{Kl},  $R$ is a Koszul ring if and only if $K$ is a Koszul  $R$-module. We have:

\medskip
\centerline{$R $ is a Koszul  ring $\iff $ $K$  is a Koszul $R$-module   $\iff$  $K$ is a Koszul $G$-module.} 
\medskip

If $R$ is a graded $K$-algebra,   Corollary \ref{finite ld(k)}  in particular says  that  $R$ is  a  Koszul ring  if $ K$ has  finite linear defect  or equivalently  $K$ is  a  Koszul module.  By \cite[1.13]{HI} and  \cite[3.4]{IR}  one gets the following result, that  is the analogous of  Theorem \ref{AEP}, with the regularity replaced by the linear defect.  

 \begin{proposition} \label{infty}
Let $R$ be a graded $K$-algebra. The following facts are equivalent: 
 \begin{itemize}
 \item[(1)]  $R$ is Koszul,
 \item[(2)]  $\ld_R(K)=0$,
 \item[(3)]  $\ld_R(K)< \infty$,
 \item[(4)]  there exists a Koszul Cohen-Macaulay $R$-module $M$ with $\mu(M)=e(M)$,
 \item[(5)]  every  Cohen-Macaulay $R$-module $M$ with $\mu(M)=e(M) $ is Koszul.
 \end{itemize} 
 \end{proposition} 
 
In \cite{IR} the modules  verifying  $\mu(M)=e(M)$ are named  modules of minimal degree.  When $R$ itself is Cohen-Macaulay, the maximal Cohen-Macaulay modules of minimal degree are precisely the so-called Ulrich modules. Cohen-Macaulay modules of minimal degree exist over any local ring, for example the residue field is one.  
 
The following question appears in \cite[1.14]{HI}. 

\begin{question}
\label{loc1}  Let $R$ be a local ring. If $\ld_R(K)<\infty$,  then is $\ld_R(K)=0$?
\end{question}

 To answer \ref{loc1} one has to  compare $\lin(\FF^R_K) $ and $\lin(\FF^G_K)$. From a minimal free resolution of $K$ as a $G$-module we can build up a free resolution (not necessarily minimal) of $K$ as an $R$-module. In some cases the  process for getting the minimal free resolution is under control via special cancellations (see  \cite[3.1]{RS}), but in general it is a difficult task. 

We may define absolutely  Koszul local rings exactly as in the graded case. A positive answer to  Question \ref{loc1} would give a positive answer to the following:

\begin{question}
\label{loc2} Let $R$ be  an absolutely Koszul local ring. Is $R$   Koszul? 
\end{question}

\end{document}